\pdfoutput=1
\RequirePackage{ifpdf}
\ifpdf 
\documentclass[pdftex]{sigma}
\else
\documentclass{sigma}
\fi

\usepackage{mathrsfs,bbm}

\numberwithin{equation}{section}

\newtheorem{Theorem}{Theorem}[section]
\newtheorem{Corollary}[Theorem]{Corollary}
\newtheorem{Lemma}[Theorem]{Lemma}
\newtheorem{Proposition}[Theorem]{Proposition}
 { \theoremstyle{definition}
\newtheorem{Definition}[Theorem]{Definition}

\newtheorem{Remark}[Theorem]{Remark}
}

\newcommand{\e}{\mathrm{e}} 
\newcommand{\N}{\mathbb{N}}
\newcommand{\R}{\mathbb{R}}
\newcommand{\Z}{\mathbb{Z}}
\newcommand{\C}{\mathbb{C}}

\renewcommand{\P}{\mathbb{P}}
\newcommand{\E}{\mathbb{E}}
\newcommand{\dd}{\mathrm{d}} 
\newcommand{\eps}{\varepsilon}
\newcommand{\la}{\langle}
\newcommand{\ra}{\rangle}

\newcommand{\be}{\begin{equation}}
\newcommand{\ee}{\end{equation}}

\newcommand{\vect}[1]{\boldsymbol{#1}}

\newcommand{\set}[1]{{\left \{ #1 \right \}}}

\newcommand{\one}{\mathbbm{1}}
\newcommand{\1}{\mathbbm{1}}

\begin{document}

\allowdisplaybreaks

\newcommand{\arXivNumber}{2311.08763}

\renewcommand{\PaperNumber}{046}

\FirstPageHeading

\ShortArticleName{Intertwinings for Continuum Particle Systems: an Algebraic Approach}

\ArticleName{Intertwinings for Continuum Particle Systems:\\ an Algebraic Approach}

\Author{Simone FLOREANI~$^{\rm a}$, Sabine JANSEN~$^{\rm bc}$ and Stefan WAGNER~$^{\rm bc}$}

\AuthorNameForHeading{S.~Floreani, S.~Jansen and S.~Wagner}

\Address{$^{\rm a)}$~Institute for Applied Mathematics, University of Bonn, Bonn, Germany}
\EmailD{\href{mailto:sflorean@uni-bonn.de}{sflorean@uni-bonn.de}}

\Address{$^{\rm b)}$~Mathematisches Institut, Ludwig-Maximilians-Universit\"at, 80333 M\"unchen, Germany}
\EmailD{\href{mailto:jansen@math.lmu.de}{jansen@math.lmu.de}, \href{mailto:swagner@math.lmu.de}{swagner@math.lmu.de}}

\Address{$^{\rm c)}$~Munich Center for Quantum Science and Technology (MCQST),\\
\hphantom{$^{\rm c)}$}~Schellingstr.~4, 80799 M\"unchen, Germany}

\ArticleDates{Received November 16, 2023, in final form May 22, 2024; Published online June 05, 2024}

\Abstract{We develop the algebraic approach to duality, more precisely to intertwinings, within the context of particle systems in general spaces, focusing on the $\mathfrak{su}(1,1)$ current algebra. We introduce raising, lowering, and neutral operators indexed by test functions and we use them to construct unitary operators, which act as self-intertwiners for some Markov processes having the Pascal process's law as a reversible measure. We show that such unitaries relate to generalized Meixner polynomials. Our primary results are continuum counterparts of results in the discrete setting obtained by Carinci, Franceschini, Giardin\`a, Groenevelt, and Redig (2019).}


\Classification{60J25; 60K35; 82C22; 22E60}

\section{Introduction}

In recent years the algebraic approach to duality has been developed in the context of interacting particle systems (see, e.g., \cite{sturm2020algebraic} for a review). Stochastic duality is a tool that connects two Markov processes via an observable of both processes (the duality function): The relation then tells that the expected evolution of such observable with respect to one Markovian dynamics is equal to the expected evolution with respect to a second one, for any time and initial conditions. When the two processes have the same semigroup, we speak of self-duality. Duality becomes relevant when one of the two processes is easier to be studied and the duality function is a meaningful observable for the other process. In the particular case of self-duality, the simplification may arise by a simpler initial condition in the dual process.

 Stochastic duality has been used in various contexts, such as interacting particle systems (see, e.g., \cite{liggett_interacting_2005}), population genetics models (see, e.g., \cite{etheridige06}), (stochastic) partial differential equations (see, e.g., \cite{Mueller15}). However, finding duality relations is not easy. The algebraic approach provides
a structured way to find duality functions; for several processes, duality relations have been found only via the algebraic approach (see, e.g., \cite{carinci2016asymmetric,carinci3367generalized,kuan2016stochastic,kuan2018algebraic}).
So far, the algebraic approach to duality has been developed only in the case of particles hopping on discrete spaces.
The basic idea of the algebraic approach is as follows: Write the generator of the Markov process as a sum of single edge-generators describing the dynamics of particles among the edges of the underlying graph. Then, identify an underlying Lie algebra such that the single edge-generators are elements of the universal enveloping algebra in a given representation. One can then exploit the commutation relations of the Lie algebra to find symmetries of the generator.
This procedure was applied by Carinci, Franceschini, Giardin\`a, Groenevelt and Redig in the context of self-duality functions for independent random walkers, the symmetric exclusion process and the symmetric inclusion process. For the latter process, a self-duality relation in terms of Meixner polynomials was recovered, which first appeared in \cite{franceschini2019stochastic}.

In fact, the very notion of self-duality in the continuum is not trivial and only recently a~generalization for particles on a Polish space, where the concept of self-duality is replaced by the one of self-intertwining, has been developed by the authors together with F.~Redig.
In~\cite{IntertwiningConsistent}, we formulated self-intertwining relations with respect to infinite-dimensional orthogonal and falling factorial polynomials. In particular, we used the modern language of point processes, see, e.g.,~\cite{last2016}, identifying the particle system with an evolving counting measure and intriguing relations were found with the literature of chaos decomposition and extended Fock spaces in infinite dimensions (see, e.g., \cite{lytvynov2003JFA}).

In the present article, we extend the algebraic approach to the continuum, in the context of the~$\mathfrak{su}(1,1)$ algebra. To that aim we first introduce raising, lowering and neutral operators indexed by test functions rather than (lattice) sites and use them to define a family of unitary operators~${\rm U}(\xi,\theta)$, indexed by parameters $\xi\in \C$ and $\theta \in \R$, in a suitable $L^2$ space of finite counting measures (Section~\ref{sec:su11}). The reference measure $\rho_{p,\alpha}$ is the law of a \emph{Pascal point process} or \emph{negative binomial process} (Section~\ref{sec:pascaldef}). It is the continuum counterpart to the product of negative binomial laws, one for each lattice site, for interacting particle systems on a lattice. Families of operators indexed by functions are standard, e.g., for the canonical commutation relations for bosons in quantum many-body mechanics or quantum field theory \cite[Section~X.7]{reedSimonII}; they also appear in connection with current algebras and quantum probability \cite{accardi2009quantum,araki1969}. We leave a~thorough analysis of the algebraic setting to a companion article \cite{algebraic_Floreani} but mention already that our raising and lowering operators are closely related to operators studied for infinite-dimensional orthogonal polynomials \cite{lytvynov2003JFA} and to representations of the algebra of the square of white noise, the current algebra of $\mathfrak{sl}(2,\R)$,
and the finite difference algebra \cite{accardi2002renormalized,boukas1991,sniady2000}.

We prove that the unitaries ${\rm U}(\xi,\theta)$ belong to the symmetry group of consistent Markov processes that admit the law of the Pascal process as a reversible measure (Theorem~\ref{thm:symmetry}). Put differently, each of these unitaries is a self-intertwiner for the semigroup: ${\rm U}(\xi,\theta) P_t = P_t {\rm U}(\xi,\theta)$. \emph{Consistency} roughly means that random removal of a particle and time evolution commute \cite{carinci2021consistent}. Theorem~\ref{thm:symmetry} applies in particular to the generalized symmetric inclusion process and thereby generalizes to the continuum item~1\,(i) in \cite[Theorem~3.1]{carinci2019orthogonal}.

For a concrete choice of the parameters $\xi$ and $\theta$, the unitary ${\rm U}(\xi,\theta)$ maps functions supported on $n$-particle configurations to generalized Meixner polynomials of degree~$n$ (Theorem~\ref{thm:unitary-meixner}), up to proportionality constants. This result generalizes item~1\,(ii) in \cite[Theorem~3.1]{carinci2019orthogonal}.
As a by-product to Theorems~\ref{thm:symmetry} and~\ref{thm:unitary-meixner}, we obtain a new algebraic proof of a self-intertwining relation proven by F.~Redig and us in \cite{IntertwiningConsistent}, see Corollary~\ref{cor:meixner-intertwining}. Our results are complemented by some considerations on how to rewrite infinitesimal generators with our raising, lowering and neutral operators (Section~\ref{sec:alg-gen}).

We treat only systems with finitely many particles. However, let us briefly remark on Markov processes with infinitely many particles. Usually the duality is between on the one hand a process with possibly infinitely many particles and on the other hand the same process but with finitely many particles \cite[Section~III.4]{liggett_interacting_2005}. Accordingly our unitary self-intertwining is replaced with a~unitary operator from a Hilbert space for \emph{finitely} many particles onto a Hilbert space for infinitely many particles. The first Hilbert space may be chosen as an extended Fock space, the second Hilbert space is still an~$L^2$ space with respect to a Pascal law with infinite mass $\alpha(E) = \infty$, and the unitary operator is the isomorphism from chaos decompositions for the Pascal point process \cite[Corollary~5.3]{Lytvynov2003}. Such an intertwining relation is proven for sticky Brownian motions in \cite[Theorem~3.5]{wagner2023infinitely-many}.

We should emphasize that our method also applies to consistent Markov processes that have a Poisson law as a reversible measure instead of a Pascal law. The relevant algebra is the Heisenberg algebra, the raising, lowering and neutral operators are replaced with creation, annihilation and number operators, and Charlier polynomials take the place of Meixner polynomials. The outcome is a generalization of item~3 in Theorem~3.1 from Carinci et al.~\cite{carinci2019orthogonal}. Creation and annihilation operators in the continuum are well known in the context of many-body quantum mechanics~\cite{reedSimonII}; this is why we focus our presentation on~$\mathfrak{su}(1,1)$.

In contrast, our method does not apply to the ${\rm SU}(2)$ symmetry and Krawtchouk polynomials relevant for exclusion processes \cite{carinci2019orthogonal}. We use a reference measure on configurations with respect to which raising and lowering operators are dual, and that reference measure should be infinitely divisible. Poisson and negative binomial laws are infinitely divisible and thus have natural L{\'e}vy processes or fields as continuum counterparts. Bernoulli and binomial laws are not infinitely divisible and it is not clear what the associated continuum random field should be.

The article is organized as follows. We introduce the space of finite counting measures, the Pascal law, and the lowering, raising and neutral operators and then we present our main results. All this is done in Section~\ref{sec:setting-results}. In Section~\ref{sec:pascal}, we give a Papangelou kernel for the Pascal point process and apply it to prove that raising and lowering operators are adjoint to each other. Theorems~\ref{thm:symmetry} and~\ref{thm:unitary-meixner} are proven in Sections~\ref{sec:symmetry} and~\ref{sec:generating}, respectively. Finally, Appendix~\ref{app:meixner} gathers some known formulas on univariate Meixner polynomials.

\section{Setting and main results} \label{sec:setting-results}

In the following, Cartesian products are always equipped with product $\sigma$-algebras. For finite interacting particle systems on lattices, the configuration space is $\N_0^\Lambda = \set{(n_x)_{x\in \Lambda}\colon n_x \in \N_0}$ with $\Lambda\subset\Z^d$. We replace $\N_0^\Lambda$ with a space of finite counting measures:
Let $(E,\mathcal E)$ be a Borel space, for example, $E= \R^d$ or more generally a Polish space with its Borel $\sigma$-algebra. Let~$\mathbf N_{<\infty}$ be the space of finite counting measures on $E$. Every element $\eta \in \mathbf N_{<\infty}$ is either zero or a finite sum $\eta= \delta_{x_1}+\cdots + \delta_{x_n}$ of Dirac measures, where $x_1, \ldots, x_n \in E$, $n \in \N$. The space $\mathbf N_{<\infty}$ is equipped with the $\sigma$-algebra $\mathcal N_{<\infty}$ generated by the counting variables $\eta \mapsto \eta(B)$, $B\in \mathcal E$.
For background on point processes and counting measures, we refer the reader to Last and Penrose~\cite{LastPenroseLecturesOnThePoissonProcess}.

\subsection{Symmetric inclusion process}

We are interested in Markov processes with state space $\mathbf N_{<\infty}$. As a guiding example we take the process with formal generator
\begin{equation} \label{eq:gsip}
	L f(\eta) = \iint \bigl( f(\eta-\delta_x+\delta_y)- f(\eta)\bigr) c(x,y) \eta(\dd x) \bigl( \alpha+\eta\bigr) (\dd y),
\end{equation}
where $\alpha$ is a finite measure on $E$ and $c\colon E\times E\to \R_+$ is a bounded measurable function with $c(x,y) = c(y,x)$ on $E^2$. In \cite{IntertwiningConsistent} we called this process \emph{generalized inclusion process}. When $E$ is a~finite set, we may identify finite counting measures $\eta$ with vectors $\vect n = (n_x)_{x\in \Lambda}\in \N_0^\Lambda$, integrals over $E$ turn into sums, and the generator turns into
\begin{equation} \label{eq:sip}
	Lf(\vect n) = \sum_{x,y\in E} \bigl( f(\vect n- \vect e_x+\vect e_y) - f(\vect n)\bigr) c(x,y) n_x (\alpha_y+n_y),
\end{equation}
where $\vect e_x(y) = \delta_{x,y}$ and $\alpha_y = \alpha(\{y\})$. This is the generator of the inhomogeneous \emph{symmetric inclusion process} (see, e.g., \cite[equation~(2.2)]{floreani2020orthogonal}), which first appeared in the homogeneous case as a dual process to a model for energy and momentum transport, see \cite{giardina-redig-vafayi2010} and references therein.

When $c(x,y)$ is the constant function with value $1$, the process corresponds to the Moran process from mathematical population genetics--classical Moran model \cite{TransitionMoran,Moran} when $E$ is finite, measure-valued Moran model (see \cite[Section 2.6]{DawsonMeasureValuedMarkovProcesses} or \cite[Section~5.4]{Etheridge00}) when $E$ is uncountable. The Moran model describes a population in which individuals have genetic types $x\in E$ that may mutate or evolve by sampling replacement (death of an individual followed by replacement with the offspring of another individual).

For the homogeneous symmetric inclusion process with finite state space / Moran model with finite type space, the correspondence was already noticed in \cite[Section~5]{carinci2015dualities}. For uncountable state spaces, we notice that the $n$-particle dynamics of the process generated by \eqref{eq:gsip} (with $c(x,y) \equiv 1$) coincides with the $n$-particle Moran model \cite[equation~(2.5.2)]{DawsonMeasureValuedMarkovProcesses} with so-called \emph{mutation operator} $A\varphi(x) = \int( \varphi(y) - \varphi(x)) \alpha(\mathrm d x)$. Our process for counting measures is similar to the measure-valued Moran model. The only difference is in bookkeeping: In mathematical population genetics, the population is often modeled not with $\eta \in \mathbf N_{<\infty}$ but rather with the empirical measure $\frac{\eta}{\eta(E)}$. This is helpful for scaling limits in which the population size goes to infinity, notably Fleming--Viot limits, which are important in population genetics (see \cite{shigaDiffusions} and references therein).

\subsection{Pascal point process} \label{sec:pascaldef}
The symmetric inclusion process with generator~\eqref{eq:gsip} has a family of reversible measures $\rho_{p,\alpha}$ indexed by $p\in (0,1)$ \cite[Theorem 5.2]{IntertwiningConsistent} where $\rho_{p,\alpha}$ is the law of the Pascal point process or negative binomial process (see, e.g., \cite[Section~2.7]{serfozo1990point}, \cite{kozubowski2008distributional} and \cite[Section~5.2]{IntertwiningConsistent}). We recall the definition here.
For $a\in \R$ and $n\in \N$, the rising factorial (also known as Pochhammer symbol)~is
\[
	(a)_0 = 1,\qquad (a)_n = a(a+1)\cdots (a+n-1).
\]

\begin{Definition}
 Let $p\in (0,1)$ and $\alpha$ be a finite measure on $E$. $\rho_{p,\alpha}$ is the uniquely defined probability measure on $\mathbf N_{<\infty}$ such that:
 \begin{itemize}\itemsep=0pt
 	\item For all $B\in \mathcal E$, and $n\in \N_0$
 	\begin{equation} \label{eq:negbin}
 		\rho_{p,\alpha}(\{\eta\colon \eta(B) =n\}) = (1-p)^{\alpha(B)} \frac{p^n}{n!} (\alpha(B))_{n}.
 	\end{equation}
 	\item For all $n \geq 2$ and all disjoint $B_1,\ldots, B_n\in \mathcal E$, the variables $\eta\mapsto \eta(B_i)$, $i=1,\ldots, n$, are independent.
 \end{itemize}
\end{Definition}

We work in the Hilbert space $\mathcal H = L^2(\mathbf N_{<\infty}, \mathcal N_{<\infty}, \rho_{p,\alpha})$ of complex-valued square-integrable functions. The scalar product is
\[
	\langle f,g \rangle = \int \overline{f} g\, \dd\rho_{p,\alpha}.
\]

\subsection[Representation of the su(1,1)]{Representation of the $\boldsymbol{\mathfrak{su}(1,1)}$ current algebra} \label{sec:su11}

Let $\mathcal D\subset \mathcal H$ be the set of bounded measurable functions $F$ for which there exists a cutoff $ m_F \in \N$ such that $F$ is supported in $\{\eta\colon \eta(E)\leq m_F\}$. We define operators $k^\pm(\varphi),k^0(\varphi)\colon\mathcal D\to \mathcal D$ indexed by bounded measurable functions $\varphi\colon E\to \C$ as follows:
\begin{align*}
	&k^+(\varphi) F(\eta) = \frac{1}{\sqrt p} \int \varphi(x) F(\eta-\delta_x) \eta(\dd x),\\
 &k^-(\varphi) F(\eta) = \sqrt{p} \int \overline{\varphi(x)} F(\eta+\delta_x) \bigl(\alpha+\eta)\bigr(\dd x),\\
	&k^0(\varphi) F(\eta) = F(\eta) \biggl( \int \varphi\,\dd \eta+ \frac12 \int \varphi \,\dd \alpha \biggr).
\end{align*}
for $\eta \in \mathbf{N}_{<\infty}$.
The raising operator $k^+(\varphi)$ and the lowering operator $k^-(\varphi)$ map functions supported in $\{\eta\colon \eta(E) = n\}$ to functions supported in $\{\eta\colon \eta(E) =n+1\}$ and $\{\eta\colon \eta(E) =n-1\}$ respectively. The indicator $\1_{\{0\}}$ that there is no particle at all (\emph{vacuum}) is annihilated by all lowering operators: $k^-(\varphi) \1_{\{0\}}=0$.

Our operators are a representation of the $\mathfrak{su}(1,1)$ current algebra, i.e., they represent the matrix-valued maps $x\mapsto \varphi(x) k^+$, $x\mapsto \overline{\varphi(x)} k^-$, $x\mapsto \theta(x) k^0$ with
\begin{align*}
 	k^+ = \begin{pmatrix} 0 & \mathrm i \\ 0 & 0 \end{pmatrix},\qquad
	k^- = \begin{pmatrix} 0 & 0 \\ \mathrm i & 0 \end{pmatrix},\qquad
	k^0 = \begin{pmatrix} 1/2 & \hphantom{-}0 \\ 0 & -1/2 \end{pmatrix},
\end{align*}
see \cite[Section~2.2]{algebraic_Floreani}.
The following commutation relations can be checked by an explicit computation:
\begin{gather*}
	\bigl[k^-(\varphi), k^+(\theta)\bigr] = 2 k^0(\overline{\varphi}\theta),\qquad
\bigl[k^0(\theta),k^+(\varphi)\bigr] = k^+(\varphi\theta),\qquad [k^0(\theta),k^-(\varphi)] = - k^-(\overline{\varphi}\theta),
\end{gather*}
where $[T, S] := TS - ST$.
Indeed, it can be readily verified that
\begin{gather*}
 k^-(\varphi) k^+(\theta) F(\eta) = \iint \overline{\varphi(y)} \theta(x) F(\eta - \delta_x + \delta_y) \eta(\dd x) (\alpha + \eta)(\dd y) \\
 \hphantom{k^-(\varphi) k^+(\theta) F(\eta) = \iint}{}
 + F(\eta) \int \overline{\varphi(y)} \theta(y) (\alpha + \eta)(\dd y), \\
 k^+(\theta) k^-(\varphi) F(\eta) = \iint \overline{\varphi(y)} \theta(x) F(\eta - \delta_x + \delta_y) \eta(\dd x) (\alpha + \eta)(\dd y) \\
 \hphantom{k^+(\theta) k^-(\varphi) F(\eta) =\iint}{}
 + F(\eta) \int \overline{\varphi(x)} \theta(x) \eta(\dd y).
\end{gather*}
Subtracting these equations yields the commutation relation
\[
\bigl[k^-(\varphi), k^+(\theta)\bigr] F(\eta) = F(\eta) \int \overline{\varphi(x)} \theta(x) (\alpha + 2\eta)(\dd x) = 2k^0(\overline{\varphi}\theta) F(\eta).
\]
 The other relations follow similarly.

The factor $\sqrt p$ included in the definitions of $k^\pm(\varphi)$ is irrelevant for the commutation relations but it matters for the adjointness relation $\la f, k^+(\varphi) g\ra = \la k^-(\varphi) f,g\ra$ proven in Lemma~\ref{lem:adjoints}.

We leave a thorough analysis of the current algebra generated by the operators $k^\pm(\varphi)$, $k^0(\varphi)$ to another article \cite{algebraic_Floreani} and focus on the operators associated with the constant function $\varphi = \1$, equal to $1$ on all of $E$. In Lemma~\ref{lem:closure} below, we check that for every $\xi \in \C$, the operator $\frac{1}{\mathrm i} \bigl( \xi k^+(\1) - \overline{\xi} k^-(\1))$ with domain $\mathcal D$ is closable with self-adjoint closure $A$. The operator $\exp( \xi k^+(\1) - \overline{\xi} k^-(\1))$ is, by definition, the unitary $\exp(\mathrm i A)$.
For $\theta \in \R$, the unitary $\exp\bigl( 2\mathrm i \theta k^0(\1)\bigr)$ is defined as a multiplication operator, it multiplies a function $f(\eta)$ with $\exp(\mathrm i \theta (\alpha (E) + 2 \eta(E)))$.
In this way we obtain a family of unitary operators
\begin{equation}\label{eq:unitary}
	{\rm U}(\xi,\theta) = \exp\bigl( \xi k^+(\1) - \overline{\xi} k^-(\1)\bigr) \exp\bigl(2\mathrm i\theta k^0(\1)\bigr),\qquad \xi \in \C,\theta\in \R.
\end{equation}
For the definition of the unitary~\eqref{eq:unitary}, it is essential that $\alpha$ has finite total mass and that each configuration $\eta$ has finitely many particles.

\subsection{Symmetries of consistent particle processes}

Let $(P_t)_{t\geq 0}$ be the semigroup of a Markov process with state space $\mathbf N_{<\infty}$. For a measurable $f\colon \mathbf N_{<\infty} \to \R_+$, let $\mathcal A f\colon \mathbf N_{<\infty} \to \R_+$ be the function given by
\[
	\mathcal A f(\eta) = \int f(\eta-\delta_x)\eta(\dd x).
\]

\begin{Definition}
 The semigroup $(P_t)_{t\geq 0}$ is \emph{consistent} if $P_t \mathcal A f = \mathcal A P_t f$ for all measurable $f\colon \mathbf N_{<\infty}\to \R_+$ and all $t\geq 0$.
\end{Definition}

Notice that, up to questions of domains, $\mathcal A = \sqrt p\, k^+(\1)$. In Proposition~\ref{proposition: consistency implies conservative} below, we show that if a process is consistent then it is \emph{conservative}, i.e., the total number of particles is constant in time. In view of that, we could normalize $\mathcal A$ by the total mass $\eta(E)$ (which is constant in time) and read consistency as the property that the action of removing a particle uniformly at random commutes with the dynamics \cite{carinci2021consistent,IntertwiningConsistent}.

The following theorem generalizes item 1\,(i) in Theorem~3.1 in Carinci et al.~\cite{carinci2019orthogonal}.

\begin{Theorem} \label{thm:symmetry}
Let $(P_t)_{t\geq 0}$ be the semigroup of a Markov process with state space $\mathbf N_{<\infty}$. Let $p\in (0,1)$ and $\alpha$ be a finite measure on $E$. Assume that the process is consistent, and admits the Pascal law $\rho_{p,\alpha}$ as a reversible measure. Then $P_t$ commutes with all unitaries ${\rm U}(\xi,\theta)$ from~\eqref{eq:unitary}:
	\[
		P_t {\rm U}(\xi,\theta) f(\eta) = {\rm U}(\xi,\theta) P_t f (\eta)
	\]
	for all $t\geq 0$, $f\in L^2(\mathbf N_{<\infty},\mathcal N_{<\infty},\rho_{p,\alpha})$ and $\rho_{p,\alpha}$-almost all $\eta\in \mathbf N_{<\infty}$.
\end{Theorem}

The generalized symmetric inclusion process is a consistent Markov process and admits the Pascal law $\rho_{p,\alpha}$ as a reversible measure (as proved in \cite[Theorem~5.2.]{IntertwiningConsistent}) and, thus, Theorem~\ref{thm:symmetry} applies.

\subsection{Intertwining with Meixner polynomials}
Let $\mathcal P_n$ be the $L^2$-closure of the space of polynomials of degree at most $n$ in the counting variables $\eta(B)$. Thus $\mathcal P_n$ is the closed linear hull of monomials $\eta\mapsto \eta(B_1)^{n_1}\cdots \eta(B_\ell)^{n_\ell}$ with $n_1+\cdots + n_\ell \leq n$ and $B_1,\ldots, B_\ell \in \mathcal E$. For $f_n\colon E^n\to \C$ a bounded measurable function, let $I_n(f_n)\in L^2(\mathbf N_{<\infty},\mathcal N_{<\infty}, \rho_{p,\alpha})$ be given by
\begin{align*}
 I_n(f_n) := f - \text{orthogonal projection of } f \text{ onto } \mathcal P_{n-1},
\end{align*}
where $f(\eta) := \int f_n \dd \eta^{\otimes n}$.
The orthogonalized version of
\[
	\eta(B_1)^{n_1}\cdots \eta(B_\ell)^{n_\ell}
\]
with disjoint $B_1,\ldots, B_\ell$ and $n_1+\cdots+n_\ell =n$, is a product of univariate Meixner polynomials with leading coefficient $1$:
\begin{equation*} 
	I_n\bigl( \1_{B_1}^{\otimes n_1}\otimes \cdots \otimes \1_{B_\ell}^{\otimes n_\ell}\bigr) (\eta) = \prod_{i=1}^\ell \mathcal M_{n_i}(\eta(B_i);\alpha(B_i),p)
\end{equation*}
for $\rho_{p,\alpha}$-almost all $\eta$. See, e.g., Lytvynov \cite[Lemma 3.1]{lytvynov2003JFA} or \cite[Proposition~5.3]{IntertwiningConsistent}.
The definition of univariate Meixner polynomials is recalled in Appendix~\ref{app:meixner}. The letter $I_n$ reflects an analogy with Poisson--Charlier polynomials and multiple stochastic integrals with respect to compensated Poisson random measures \cite[Section 12]{LastPenroseLecturesOnThePoissonProcess}.

Our second main theorem generalizes item 1\,(ii) in Theorem 3.1 in Carinci et al.~\cite{carinci2019orthogonal}.
For $f\colon \mathbf N_{<\infty}\to \C$ we define a sequence of functions $(f_n)_{n\in \N}$ by
\[
	 f_n(x_1,\ldots, x_n) = f(\delta_{x_1}+\cdots + \delta_{x_n} \bigr).
\]
Each function $f_n$ is \emph{symmetric}, i.e., its value stays the same when the variables are permuted. The inverse of the hyperbolic tangent is denoted $\mathrm{artanh}$.

\begin{Theorem} \label{thm:unitary-meixner}
	Let $\xi= \mathrm{artanh}\, \sqrt p$ and $U= \exp\bigl(\xi k^+(\1) - \xi k^-(\1)\bigr)$. Then for all $n\in \N$ and all bounded measurable $f\colon \mathbf N_{<\infty}\to \C$ supported in
$\{\eta\colon \eta(E)=n\}$, we have
	\[
		 U f = (1-p)^{\alpha(E)/2} \, \frac{1}{n!}\, (1-p)^n I_n(f_n).
	\]
	Furthermore, $ U \1_{\{0\}} = (1-p)^{\alpha(E)/2} \, \1$.
\end{Theorem}

Combining Theorems~\ref{thm:symmetry} and~\ref{thm:unitary-meixner}, we obtain a new proof of an intertwining relation proven by other means in \cite[Theorems~3.15 and~5.2]{IntertwiningConsistent}. Let \smash{$\bigl(p_t^{(n)}\bigr)_{t\geq 0}$} be a semigroup on symmetric bounded and measurable functions from $E^n$ to $\C$ that satisfies
\[
	(P_t f)(\delta_{x_1}+\cdots + \delta_{x_n}) = \bigl(p_t^{(n)} f_n\bigr)(x_1,\ldots,x_n).
\]
for all measurable bounded $f \colon \mathbf{N}_{<\infty} \to \C$ and $x_1, \ldots, x_n \in E$.
\begin{Corollary} \label{cor:meixner-intertwining}
	Under the assumptions of Theorem~$\ref{thm:symmetry}$: For all $t\geq 0$, all $n\in \N$, and every bounded measurable $f_n\colon E^n\to \C$,
	\[
		P_t I_n(f_n) = I_n\bigl(p_t^{(n)} f_n\bigr).
	\]
\end{Corollary}

\begin{proof}
	Apply Theorem~\ref{thm:symmetry} to ${\rm U}(\xi,\theta)$ for $\xi =\mathrm{artanh}\, \sqrt p$ and $\theta =0$, combine with Theorem~\ref{thm:unitary-meixner}, and observe that $n$-dependent constants may be dropped because the process is conservative (see Proposition~\ref{proposition: consistency implies conservative} below).
\end{proof}

Furthermore, we note that, in addition to unitary intertwinings, relation \cite[equa\-tion~(3.6)]{IntertwiningConsistent} in terms of the $K$-transform can also be derived using algebraic techniques: The $K$-transform can be expressed with the operator $k^+(\one)$. More precisely, it is equal to $\exp \bigl(\sqrt{p} k^+(\one) \bigr) f(\eta)$. This relation serves as the continuum counterpart to \cite[Lemma~4.2]{carinci2021consistent}.

\subsection{Algebraic expression for the generator} \label{sec:alg-gen}

On the lattice the generator~\eqref{eq:sip} of the symmetric inclusion process can be expressed directly in terms of raising and lowering operators associated with lattice sites \cite{gkrv}: Suppose that $\alpha_y =\alpha$ does not depend on $y$. Let $\vect e_x \in \N_0^E$ be the vector given by $\vect e_x(y) = \delta_{x,y}$. Set
\begin{gather*}
	k_x^+ f(\vect n ) = \frac{1}{\sqrt p}\, n_x f(\vect n - \vect e_x),\quad
	k_x^- f(\vect n) = \sqrt p\, (\alpha + n_x) f(\vect n + \vect e_x),\\
	k_x^0 f(\vect n) = \Bigl( \frac \alpha 2 + n_x\Bigr) f(\vect n).
\end{gather*}
\big(Our operators $k^\#(\varphi)$ correspond to $k^\#(\varphi) = \!\sum_x \varphi(x) k_x^\#$ for $\#=+,0$ and $k^-(\varphi) =\! \sum_{x} \overline{\varphi(x)} k_x^-$.\big)
Then
\begin{align} \label{eq:sip-algebraic}
	L & = \notag \sum_{\substack{x,y\in E\\x \neq y}} c(x,y)\biggl( k_x^+ k_y^- + k_x^- k_y^+ - 2 k_x^0 k_y^0 + \frac{\alpha^2}{2} I \biggr) \\
 & = \sum_{x,y\in E} c(x,y)\biggl( k_x^+ k_y^- + k_x^- k_y^+ - 2 k_x^0 k_y^0 + \biggl(\frac{\alpha^2}{2} - \alpha \one_{\{x=y\}} \biggr) I \biggr),
\end{align}
where $I$ is the identity operator. The rewrite is a key ingredient to the algebraic approach for duality: Dual processes such as the Brownian energy process have the same algebraic expression but for a different representation of the $\mathfrak{su}(1,1)$ algebra \cite[Section~6.2]{gkrv}.

In the continuum, equation~\eqref{eq:sip-algebraic} is problematic as we cannot define operators $k_x^\#$--in the lingo of quantum field theory, $x\mapsto k_x^\#$ only make sense as operator-valued \emph{distributions} \cite[Sec\-tion~3.7]{talagrand-book}. However, using the operators $k^\#(\varphi)$ introduced above, equation~\eqref{eq:sip-algebraic} generalizes nicely, when $c(x,y)$ is of product form $c(x,y) = 2 \varphi(x)\varphi(y)$. In that case
\begin{equation} \label{eq:lkphi}
	L = k^+(\varphi) k^-(\varphi) + k^-(\varphi)k^+(\varphi) - 2 k^0(\varphi)^2 + \biggl( \frac{1}{2} \biggl(\int \varphi \, \dd \alpha \biggr)^2 - \int \varphi^2 \, \dd \alpha \biggr) I.
\end{equation}
This expression makes perfect sense in the continuum setup as well. Moreover, it can readily be checked that \eqref{eq:lkphi} holds if $L$ is the formal generator \eqref{eq:gsip} of the generalized symmetric inclusion process as well. A similar rewrite is possible if $c(x,y)$ is a linear combination of symmetrized products $\varphi_1(x)\varphi_2(y) + \varphi_2(x)\varphi_1(y)$. More generally, say on $E=\R^d$ with homogeneous measures $\alpha(\dd x) = \alpha \dd x$, one may hope to give meaning to expressions like
\[
	L = \int c(x,y) \biggl(k_x^+k_y^- + k_y^+ k_x^- - 2 k_x^0 k_y^0 + \frac{\alpha^2}{2} \biggr) \dd x \dd y - \alpha \int c(x,x) \, \dd x
\] 	
by first defining it as~\eqref{eq:lkphi} and its siblings when $c(x,y)$ is of product form or a linear combination of symmetrized products, and second by approximating sufficiently nice functions $c(x,y)$ by linear combinations of symmetrized products. We leave as an open problem to carry out this program, to clarify its usefulness for Markov processes, and to figure out its relation to Hamilton operators $\int \mu(k) a(k)^\dagger a(k) \dd k$ in quantum field theory and quantum many-body theory (see, e.g., the non-numbered proposition preceding Theorem~X.45 in Reed and Simon~\cite{reedSimonII} or Talagrand~\mbox{\cite[Section~3.7]{talagrand-book}}).

\section{Papangelou kernel for the Pascal point process} \label{sec:pascal}

If $X$ is a negative binomial variable with parameters $p\in (0,1)$ and $a>0$, then
\[
 (n+1)\P( X = n+1) = (1-p)^a\, \frac{(a)_{n+1}}{n!} p^{n+1} = p (a+n)\P( X = n),
\]
for all $n\in \N_0$, accordingly
\[
	\E[ X f(X)] = \E[ p(a+X) f(X+1)]
\]
for all $f\colon \N_0\to \R_+$. The following proposition gives the analogous property for Pascal point processes. For future reference we state and prove the proposition for $\sigma$-finite measures $\alpha$, allowing for configurations with infinitely many particles. Thus, let $\mathbf N$ be the space of $s$-finite counting measures on $E$. As $E$ is a Borel space, $\mathbf N$ consists of the measures $\eta$ on $E$ that are finite or countable sums of Dirac measures. The space $\mathbf N$ is equipped with the $\sigma$-algebra $\mathcal N$ generated by the counting variables $B\mapsto \eta(B)$, $B\in \mathcal E$.

A \emph{Pascal point process} with parameters $p$ and $\alpha$ is a random variable $\eta\colon(\Omega,\mathcal F, \P)\to (\mathbf N,\mathcal N)$ that satisfies the properties listed in the definition of $\rho_{p,\alpha}$: Counting variables for disjoint regions~$B_i$ are independent, and each counting variable $\eta(B)$ has a negative binomial distribution with parameters $p$ and $\alpha(B)$. For sets with infinite mass $\alpha(B) = \infty$, the counting variable $\eta(B)$ is almost surely infinite. Pascal point processes have Laplace functional
\begin{gather*}
	\E\bigl[ \e^{-\int f\dd \eta}\bigr] = \exp \biggl(\! - \int \log \biggl(\frac{1- p\,\e^{-f(x)}}{1- p}\biggr) \alpha (\dd x)\biggr) = 	\exp\Biggl(\! - \sum_{j=1}^\infty \int \bigl(1- \e^{- j f(x)}\bigr) \frac{p^j}{j} \alpha(\dd x) \Biggr)
\end{gather*}
for measurable $f \colon E \to \R_+$.

\begin{Proposition}
	Fix $p\in (0,1)$ and a $\sigma$-finite measure $\alpha$ on $E$. Let $\eta$ be a Pascal point process with parameters $p$ and $\alpha$. Then
	\[
		\E\biggl[ \int F(x,\eta)\eta(\dd x)\biggr] = \E\biggl[\int F(x,\eta+\delta_x) p (\alpha+\eta)(\dd x)\biggr]
	\]
	for all measurable $F\colon E\times \mathbf N \to \R_+$.
\end{Proposition}

The proposition says that the kernel $E\times \mathcal N\to \R_+$, $(x,B)\mapsto \kappa(x,B) = p(\alpha(B) + \eta(B))$ is a~\emph{Papangelou kernel} for the Pascal point process. For a general discussion on Papangelou kernels, we refer to \cite{MatthesWarmuthMecke} or \cite[Section 1.2.2]{Rafler2009}.

\begin{proof}
	Pascal point processes are compound Poisson; we deduce the proposition from the Mecke formula for Poisson point processes. Let $\Pi = \sum_i \delta_{(X_i,N_i)}$ be a Poisson point process on $E\times \N_0$ with intensity measure $\beta = \alpha\otimes \bigl(\sum_{n=1}^\infty \frac{p^n}{n} \delta_n\bigr)$. Then $\xi = \sum_i N_i \delta_{X_i}$ is a Pascal point process with law $\rho_{p,\alpha}$. The Mecke equation for $\Pi$ \cite[Theorem 4.1]{LastPenroseLecturesOnThePoissonProcess} and measurable $G\colon (E\times \N_0)\times \mathbf N(E\times \N_0) \to \R_+$ gives
\[
	\E\biggl[ \sum_i G( (X_i, N_i), \Pi )\biggr] = \sum_{n=1}^\infty \frac{p^n}{n} \int \E\bigl[ G\bigl( (x,n), \delta_{(x,n)} + \Pi\bigr) \bigr] \alpha (\dd x).
\]
Let $\varphi\colon E\times \mathbf N(E)\to \R_+$ be a measurable function. We apply the Mecke equation for $\Pi$ to $G\bigl((x,n),\sum_i \delta_{(x_i,n_i)}\bigr):= n \varphi\bigl(x, \sum_i n_i \delta_{x_i}\bigr)$ and get
\begin{equation}
 \label{equation: mecke1}
	\E\biggl[\sum_i N_i \varphi( X_i, \xi)\biggr] = \sum_{n=1}^\infty p^n \int \E[ \varphi(x, \xi+ n \delta_x)] \alpha(\dd x).
\end{equation}
Using \eqref{equation: mecke1} for $\varphi(x,\mu) = F(x,\mu + \delta_x)$, we obtain
\begin{align}
 p\, \E \biggl[ \sum_i N_i F\bigl( X_i, \xi +\delta_{X_i}\bigr)\biggr] &= p\sum_{n=1}^\infty p^n \int \E [ F(x, \xi + (n+1) \delta_{x}) ]\alpha(\dd x) \notag \\
 &= \sum_{n=2}^\infty p^n \int \E [ F(x, \xi + n \delta_{x}) ] \alpha(\dd x). \label{eq:mecke2}
\end{align}
We apply~\eqref{equation: mecke1} to $\varphi(x,\mu) = F(x,\mu)$ and eliminate the sum over $n\geq 2$ using~\eqref{eq:mecke2}. This gives
\begin{equation*}
	\E\biggl[\sum_i N_i \varphi( X_i, \xi)\biggr] =
		p \int \E[ F(x, \xi+ \delta_x)] \alpha(\dd x) + p\, \E \biggl[ \sum_i N_i F\bigl( X_i, \xi +\delta_{X_i}\bigr)\biggr].
\end{equation*}
Equivalently,
\[
	\E\biggl[ \int F(x,\xi) \xi(\dd x)\biggr] = p\, \E\biggl[ \int F(x, \xi+\delta_x) (\alpha + \xi)(\dd x)\biggr]
\]
and the proof is complete as $\xi$ and $\eta$ are equal in distribution.
\end{proof}

Turning back to the law $\rho_{p,\alpha}$ on $\mathbf N_{<\infty}$ for finite measures $\alpha$, we obtain
\begin{equation} \label{eq:papangelou}
	\int \biggl( \int F(x,\eta)\eta(\dd x)\biggr) \rho_{p,\alpha}(\dd x)
	= \int\biggl( \int F(x,\eta+\delta_x) p (\alpha+\eta)(\dd x)\biggr) \rho_{p,\alpha}(\dd x)
\end{equation}
for all measurable $F\colon E\times \mathbf N_{<\infty}\to \R_+$, hence also all bounded measurable functions $F$ (here we use that the expected total number of particles is finite when $\alpha$ has finite total mass).

\begin{Lemma} \label{lem:adjoints}
	For all bounded measurable $\varphi\colon E\to \C$ and all $f,g\in \mathcal D$,
	\[
		\bigl\langle f, k^+(\varphi) g\bigr\rangle = \langle k^-(\varphi) f,g \rangle.
	\]
\end{Lemma}

\begin{proof}	
	Let us denote integration with respect to $\rho_{p,\alpha}$ as $\E[f(\eta)] =\int f\dd \rho_{p,\alpha}$. We apply equation~\eqref{eq:papangelou} to $F(x,\eta) = \varphi(x) \overline{f(\eta)} g(\eta - \delta_x)$ and obtain
	\begin{align*}
		\bigl\langle f, k^+(\varphi) g \bigr\rangle
			& = \frac 1{\sqrt p} \E\biggl[\int \varphi(x) \overline{f(\eta)} g(\eta - \delta_x) \eta(\dd x)\biggr]\\
			& = \sqrt p\, \E\biggl[\int \varphi(x) \overline{f(\eta + \delta_x)} g(\eta) (\alpha+ \eta)(\dd x)\biggr]\\
			& = \langle k^-(\varphi) f,g \rangle. \qedhere
\tag*{\qed}
\end{align*}
\renewcommand{\qed}{}
\end{proof}

\section{Symmetries. Proof of Theorem~\ref{thm:symmetry}} \label{sec:symmetry}

Before we turn to the proof of Theorem~\ref{thm:symmetry}, we shows that consistency implies conservativity, which is of interest in its own.

\begin{Proposition}
 \label{proposition: consistency implies conservative}
	Every consistent Markov semigroup is conservative.
\end{Proposition}

\begin{proof}
	Let $(P_t)_{t\geq 0}$ be Markov semigroup on $(\mathbf N_{<\infty}, \mathcal N_{<\infty})$ that is consistent, i.e., $P_t \mathcal A f = \mathcal A P_t f$ for all $t\geq 0$ and all measurable $f\colon \mathbf N_{<\infty} \to \R_+$. Then, for all $k\in \N$,
	\begin{equation} \label{eq:ptak}
		P_t \mathcal A^k \one = \mathcal A^k P_t \one = \mathcal A^k \one.
	\end{equation}
	A straightforward computation shows that $\mathcal A^k\one$ is a descending factorial power of the total number of particles,
	\begin{equation} \label{eq:ak-factorial}
		\bigl( \mathcal A^k \one \bigr)(\eta) = \eta(E)( \eta(E)-1) \cdots ( \eta(E) - k +1),
	\end{equation}
	for all $\eta \in \mathbf N_{<\infty}$. Let $(\eta_t)_{t\geq 0}$ be a Markov process with semigroup $(P_t)_{t\geq 0}$ and deterministic initial condition $\eta_0=\mu \in \mathbf N_{<\infty}$, defined on some probability space $\bigl(\Omega,\mathscr F, \mathbb P_\mu\bigr)$. Equations~\eqref{eq:ptak} and~\eqref{eq:ak-factorial} show that under $\mathbb P_\mu$, at all times $t\geq 0$, the total number of particles $\eta_t(E)$ has the same factorial moments as $\eta_0(E)$. The factorial moments of $\eta_0(E) = \mu(E)$ (hence $\eta_t(E)$) vanish for $k\geq \mu(E)+1$, therefore the moment problem is uniquely solvable and we conclude $\eta_t(E) = \eta_0(E) = \mu(E)$, $\P_\mu$-almost surely, and the process is conservative.
\end{proof}

Remember that the set $\mathcal D\subset \mathcal H$ consists of the bounded measurable functions $f$ supported in $\{\eta\colon \eta(E)\leq n_f\}$ for some $n_f<\infty$.

\begin{Lemma} \label{lem:closure}
	For every $\xi \in \C$, the operator $\frac{1}{\mathrm i} \bigl( \xi k^+(\1) - \overline{\xi} k^-(\1)\bigr)$ with domain $\mathcal D$ is closable and its closure is self-adjoint.
\end{Lemma}

Put differently, the operator is \emph{essentially self-adjoint} on $\mathcal D$ and the latter is a \emph{core} for the self-adjoint closure \cite[Section VIII.2]{reedSimonI}.

\begin{proof}
	The operator $A = \frac{1}{\mathrm i} \bigl( \xi k^+(\1)- \overline{\xi} k^-(\1)\bigr)$ is similar to the Segal field operators for free quantum fields, we adapt the proof of essential self-adjointness for the latter from Reed and Simon \cite[Theorem X.41\,(a)]{reedSimonII}.
	The domain $\mathcal D$ is dense in $\mathcal H$ and satisfies $\la f, A g\ra = \la Af, g\ra$ for all $f,g\in \mathcal D$ by Lemma~\ref{lem:adjoints}. Thus, $A$ with domain $\mathcal D(A) = \mathcal D$ is symmetric. Symmetric operators are always closable. 	
	To show that the closure is self-adjoint, we check that every vector $f\in \mathcal D$ is \emph{analytic} for $A$ and conclude with Nelson's analytic vector theorem \cite[Theorem~X.39]{reedSimonII}.
 As a reminder, Nelson's analytic vector theorem states that a symmetric operator on a Hilbert space whose domain contains a total set of analytic vectors is essential self-adjoint.
 A vector~$f$ is analytic for $A$ if there exists $\eps>0$ such that $\sum_{n=0}^\infty n!^{-1} \eps^n \| A^n \varphi\|<\infty$.
	Let $f\in \mathcal D$ and $m=m_f\in \N$ be such that $|f(\eta)|\leq \|f\|_\infty \1_{\{\eta(E)\leq m_f\}}$ on $\mathbf N_{<\infty}$. Then
	\begin{align*}
		&\bigl| k^+(\1) f(\eta)\bigr| \leq \frac{1}{\sqrt p} \|f\|_\infty (m_f+1) \1_{\{\eta(E) \leq m_f+1\}},\\
		&| k^-(\1) f(\eta)| \leq \sqrt p\, \|f\|_\infty \bigl(\alpha(E)+m_f-1\bigr) \1_{\{\eta(E) \leq m_f-1\}}
	\end{align*} 	
	for all $\eta \in \mathbf N_{<\infty}$. Set $\beta:= \alpha(E)$, we obtain a common upper bound to $k^\pm(\1) f$:
	\[
		\bigl| k^\pm(\1) f(\eta)\bigr| \leq \frac{1}{\sqrt p} \|f\|_\infty (\beta +m_f+1) \1_{\{\eta(E) \leq m_f + 1\}}.
	\]
	We expand the power $\bigl(k^+(\1) - k^-(\1)\bigr)^n f$ and bound each term in the expansion by repeated use of the previous inequality. This gives
	\[
		\bigl| \bigl(k^+(\1) - k^-(\1)\bigr)^n f(\eta)\bigr| \leq \biggl(\frac{ 2 \| f\|_\infty}{\sqrt p}\biggr)^{n} ( \beta + m_f + 1)_n \1_{\{\eta(E)\leq m_f+n\}}.
	\] 	
	The indicator on the right has norm $\leq 1$. Set $s = 2|\xi|\, \|f\|_\infty/\sqrt p$, we obtain
	\[
		\sum_{n=0}^\infty \frac{\eps^n}{n!} \bigl\| \bigl(\xi k^+(\1) -\overline{\xi} k^-(\1)\bigr)^n f \bigr\|
			 \leq \sum_{n=0}^\infty \frac{(s\eps)^n}{n!} \bigl(\beta + m_f + 1)_n,
	\] 	
	which is finite for $\eps< 1/s$. Thus, every vector $f\in \mathcal D$ is analytic and $A$ is essentially self-adjoint on $\mathcal D$. 		
\end{proof}

\begin{proof}[Proof of Theorem~\ref{thm:symmetry}]
	Let $t \geq 0$. As $\rho_{p,\alpha}$ is reversible, $P_t$ is a well-defined, self-adjoint and bounded operator on $L^2(\mathbf N_{<\infty}, \mathcal N_{<\infty}, \rho_{p, \alpha})$. The operator $\exp\bigl(2 \mathrm i \theta k^0(\1)\bigr)$ acts as multiplication with $\exp( \mathrm i \theta (\alpha(E) + 2\eta(E)))$, it commutes with $P_t$ because $(P_t)_{t \geq 0}$ preserves the total number of particles (see Proposition~\ref{proposition: consistency implies conservative} below).
	
		For $\exp\bigl( \xi k^+(\1) - \overline{\xi} k^-(\1)\bigr)$, the short informal reasoning is simple: $P_t$ is self-adjoint and commutes with $k^+(\1)$, therefore it also commutes with the adjoint $k^-(\1)$ and also with the difference $\xi k^+(\1) - \overline{\xi} k^-(\1)$ and the exponential $\exp\bigl( \xi k^+(\1) - \overline{\xi} k^-(\1)\bigr)$. The precise reasoning is slightly longer because the operators are unbounded and we have to consider domains carefully. Let $f$, $g$ be non-negative functions in $\mathcal D$. Arguing as in the proof of Lemma~\ref{lem:adjoints}, we obtain $\langle k^-(\1) f, P_t g\rangle = \langle f, \mathcal A P_t g \rangle$. Therefore,
	\[
		\langle k^-(\1) f, P_t g\rangle = \langle f, \mathcal A P_t g \rangle = \langle f, P_t \mathcal A g\rangle = \bigl\langle P_t f, k^+(\1) g\bigr\rangle
	\]
	for all non-negative $f,g\in \mathcal D$. 	Taking complex conjugates and switching $f$ and $g$, we get \smash{$\bigl\langle k^+(\1) f, P_t g \bigr\rangle = \langle P_t f, k^-(\1) g\rangle $} and then, by linear combination,
	\[
		\biggl \langle \frac 1{\mathrm i} \bigl(\xi k^+(\1) - \overline{\xi} k^-(\1)\bigr) f, P_t g\biggr \rangle =
		\biggl \langle P_t f, \frac 1{\mathrm i} \bigl(\xi k^+(\1) - \overline{\xi} k^-(\1)\bigr) g\biggr \rangle
	\]
	for all $f,g\in \mathcal D$. By Lemma~\ref{lem:closure}, the operator $- \mathrm {i}\bigl(\xi k^+(\1) - \overline{\xi} k^-(\1) \bigr)$ with domain $\mathcal D$ is closable and its closure $A$ is self-adjoint. 	
	Taking limits in the above equation, we obtain
	\[
		\langle A f, P_t g \rangle = \langle P_t f, A g\rangle
	\]
	for all $f$, $g$ in the domain of $A$. In particular, for fixed $g$ in the domain of $\mathcal A$, and all $f\in \mathcal D(A)$, we have $\langle A f, P_t g\rangle = \langle f, P_t A g \rangle$, hence $P_tg$ is in the domain of $A^*$ and $A^* P_t g = P_t A g$.
	
	As $A$ is self-adjoint, we conclude that $P_t$ leaves the domain of $A$ invariant and $P_t A = A P_t$ on the domain of $A$.
It follows that $P_t$ commutes with all spectral projections of $A$ \cite[Proposition~5.26]{schmudgen2012unbounded} and then, by spectral calculus, $P_t \exp(\mathrm i A) = \exp(\mathrm i A) P_t$. 	
\end{proof}

\section[Unitary operator vs.\ Meixner polynomials. Proof of Theorem~\ref{thm:unitary-meixner}]{Unitary operator vs.\ Meixner polynomials.\\ Proof of Theorem~\ref{thm:unitary-meixner}}
\label{sec:generating}

Here we prove Theorem~\ref{thm:unitary-meixner}. Our proof is similar in spirit to the proof with generating functions from Carinci et al.~\cite{carinci2019orthogonal}, see in particular Proposition~\ref{prop:meixner} below. In addition, we point out a~connection with the representation of the group ${\rm SU}(1,1)$ with the M{\"o}bius transform exploited by Bargmann~\cite[Section 9]{bargmann1947}.

For measurable $z\colon E\to \C$ with $\sup |z|<1$, we define the \emph{exponential state}
\[
	\mathcal E_z(0) = 1, \qquad \mathcal E_z\bigl(\delta_{x_1}+\cdots + \delta_{x_n}\bigr) = \frac{1}{\sqrt p^n} \, z(x_1)\cdots z(x_n).
\]
The exponential state owes its name to the relation $\mathcal E_z = \exp\bigl( k^+(z)\bigr)\1_{\{0\}}$, compare Meyer \mbox{\cite[Sec\-tion IV.1]{meyer2006quantum}}. The state has norm
\[
	\|\mathcal E_z\|^2 = (1-p)^{\alpha(E)} \exp\biggl( - \int \log \bigl(1-|z|^2\bigr) \dd \alpha \biggr),
\]
following analogous arguments as in \cite[equation~(5.7)]{IntertwiningConsistent}. For $\xi\in \C$, consider the matrix
\[
	A(\xi) = \exp\biggl( \begin{pmatrix} 0 & \mathrm i \xi \\ - \mathrm i \overline{\xi} &0 \end{pmatrix}\biggr)
	= \begin{pmatrix} \cosh |\xi| & \mathrm i \dfrac{\xi}{|\xi|} \sinh |\xi| \\
	- \mathrm i \dfrac{\overline{\xi}}{|\xi|} \sinh |\xi| & \cosh |\xi|\end{pmatrix}.
\]
The matrix is in ${\rm SU}(1,1)$, the group of $2\times 2$ matrices of the form
\[
	A = \begin{pmatrix} a & b \\ \overline b & \overline a \end{pmatrix}
\]
with $a,b\in \C$ and $|a|^2 - |b|^2 = 1$. It is known that ${\rm SU}(1,1)$ acts on the open complex unit disk $\{z\colon|z|<1\}$ by the M{\"o}bius transform
\[
	\phi_A(z) = \frac{a z + b}{\overline b z + \overline a}.
\]
For a function $z$ from $E$ to the open unit disk, we define a new function $z_\xi$ by
\[
	z_\xi(x) = \frac 1 {\mathrm i} \phi_{A(\xi)}\bigl(\mathrm i z(x)\bigr)
	= \frac{z(x) + \frac{\xi}{|\xi|} \tanh |\xi|}{1+ z(x) \frac{\overline{\xi}}{|\xi|} \tanh |\xi| }
\]
and a scalar $C(\xi) \in \mathbb C$ by
\[
	C(\xi) = \exp\biggl( - \int \log\left( \cosh |\xi| + z(y) \frac{\overline \xi}{\xi} \sinh |\xi|\right) \alpha(\dd y) \biggr).
\]
The logarithm is the principal branch of the complex logarithm on $\C \setminus (-\infty,0]$, i.e., $\log z = \log |z| +\mathrm i \theta$ with $\theta\in (-\pi,\pi)$ an argument of $z$. More concretely, in view of $|z(y)|<1$ and $\tanh|\xi|<1$,
\[
	\log\left( \cosh |\xi| + z(y) \frac{\overline \xi}{\xi} \sinh |\xi|\right)
	= \log (\cosh|\xi|) + \log \left( 1 + z(y) \frac{\overline \xi}{\xi} \tanh |\xi|\right)
\]
with $\log(1+u) = \sum_{n=1}^\infty (-1)^{n-1} u^n/n$.

The scalar $C(\xi)$ is similar to multipliers introduced by Bargmann when he constructed group actions ${\rm SU}(1,1)$ on spaces of functions from the complex unit disk to $\C$ \cite[Section 9]{bargmann1947}.

\begin{Proposition} \label{prop:expo-trafo}
	Let $\xi \in \C$ and $z\colon\E\to \C$ be measurable with $\sup |z|<1$.
	Then
	\[
		\exp\bigl( \xi k^+(\1) - \overline \xi k^-(\1)\bigr) \mathcal E_z = C(\xi) \mathcal E_{z_\xi}.
	\]
\end{Proposition}

\begin{proof}
	To lighten notation we write down the proof for real-valued $\xi$ only, the general case is similar. Thus let $\xi = t\in \R$. Set $f_t:= C(t) \mathcal E_{z_t}$. We show that (i) $t\mapsto f_t$ is norm differentiable, (ii) $ f_t$ is in the domain of the closure $A$ of $\frac 1{\mathrm i } \bigl(k^+(\1) - k^-(\1)\bigr)$, and (iii) $\partial_t f_t = \mathrm i A f_t$. The key equations are~\eqref{eq:ctdiff}, \eqref{eq:ezdiff}, \eqref{eq:ftdiff}, and~\eqref{eq:kpmft}. Items (i)--(iii) imply $f_t = \exp(\mathrm i t A) f_0$ for all $t\in\R$. The proposition follows upon specializing to $t = \xi$.

(i) The map $t\mapsto z_t$ is differentiable,
\[
	\partial_t z_t(x) = \frac{\bigl(1-z_0(x)^2\bigr) \bigl(1- \tanh(t)^2\bigr)}{(1+z_0(x)\tanh t)^2} = 1- z_t(x)^2
\]
pointwise for all $x$ and actually in supremum norm
\[
	\lim_{h\to 0} \biggl\| \frac{1}{h} (z_{t+h} - z_t) - (1- z_t)^2 \biggr\|_\infty =0.
\]
The multiplier $C(t)$ is differentiable with derivative
\begin{align}
	C'(t) &= - C(t) \frac{\dd}{\dd t}\int ( \log (\cosh t) + \log (1+z_0(y)\tanh t)) \alpha (\dd y) \nonumber\\
	 &= - C(t) \int z_t(y) \alpha(\dd y).\label{eq:ctdiff}
\end{align}
To prove that the map $t\mapsto \mathcal E_{z_t}$ is differentiable, we express $\mathcal E_{z_{t+h}} - \mathcal E_{z_t}$ with $t,h\in \R$ in terms of $v_{t,h}:= z_{t+h} - z_t$. For the empty configuration $\eta=0$, we note that both $\mathcal E_{z_{t+h}}(0)$ and $\mathcal E_{z_t}(0)$ are equal to $1$ hence the difference vanishes. For single-particle configurations $\eta=\delta_{x_1}$, we have $\mathcal E_{z_{t+h}} (\delta_{x_1}) - \mathcal E_{z_{t}} (\delta_{x_1}) = v_{t,h}(x_1)$.
For $n\geq 2$, we have
\begin{align}
	&\bigl(\mathcal E_{z_{t+h}} - \mathcal E_{z_t}\bigr) (\delta_{x_1}+\cdots+\delta_{x_n}) \nonumber\\
	&	\qquad {} = \frac{1}{{\sqrt p}^n} \Biggl( \prod_{i=1}^n \bigl(z_t(x_i)+ v_{t,h}(x_i)\bigr) - \prod_{i=1}^n z_t(x_i)\Biggr) \nonumber \\
	& \qquad {} = \frac{1}{{\sqrt p}^n} \sum_{i=1}^n v_{t,h}(x_i) \prod_{j \neq i} z_t(x_j) + \frac{1}{{\sqrt p}^n} \sum_{\substack{I\subset [n]:\\ |I|\geq 2}} \prod_{i\in I} v_{t,h}(x_i) \prod_{j\notin I} z_t(x_j). \label{eq:rev2}
\end{align}
We define a function $Q_{t,h}\colon \mathbf N_{<\infty}\to \C$ as follows: If $\eta(E)\leq 1$, then $Q_{t,h}(\eta):=0$; if $\eta(E) \geq 2$, write $\eta = \delta_{x_1}+\cdots+\delta_{x_n}$ with $n\geq 2$ and $x_1,\ldots,x_n\in E$ and set
\[
	Q_{t,h}(\delta_{x_1}+\cdots + \delta_{x_n}) := \frac{1}{{\sqrt p}^n} \sum_{\substack{I\subset [n]:\\ |I|\geq 2}} \prod_{i\in I} v_{t,h}(x_i) \prod_{j\notin I} z_t(x_j)
\]
(this is precisely the second term in~\eqref{eq:rev2}). Let us check that $\|Q_{t,h}\| = O\bigl(h^2\bigr)$ as $h\to 0$. At fixed number of particles $n\geq 2$, $Q_{t,h}$ is bounded by
\[
	\sup_{\eta\colon   \eta(E) = n} |Q_{t,h}(\eta)| \leq \frac{1}{{\sqrt p}^{n}} \binom{n}{2} \|v_{t,h}\|_\infty^2 \|z_t\|^{n-2}.
\]
Therefore, the $L^2$-norm satisfies
\[
	\|Q_{t,h}\|^2 \leq \sum_{n\geq 2} \rho_{p,\alpha}(\eta(E) = n) \sup_{\eta\colon   \eta(E) = n}|Q_{t,h}(\eta)|^2
	 \leq \kappa \|v_{t,h}\|_\infty^4
\]
with
\[
	\kappa = (1-p)^{\alpha(E)} \sum_{n\geq 2} \frac{(\alpha(E))_n}{n!} n^4 \|z_t\|_\infty^{2n-4}.
\]
For proving $\kappa <\infty$ it is enough to show that $\|z_t\|_\infty <1$. By definition of $z_t$, the set $\{z_t(x)\colon {x\in E}\}$ is contained in the image of the disk $D_0:=\{\zeta \in \mathbb C\colon |\zeta|\leq \|z_0\|_\infty\}$ under the M{\"o}bius transform $\zeta \mapsto (\zeta + \tanh t)/(1+ \zeta \tanh t)$. It is known that the M{\"o}bius transform maps the open unit disk to itself, and it maps compact sets to compact sets because it is continuous. Furthermore, $D_0$ is compact and contained in the open unit disk (because $\|z_0\|_\infty <1$). Therefore, the M{\"o}bius transform maps $D_0$ to a compact subset $K_t$ of the open unit disk. In particular,~$K_t$ stays away from the boundary of the unit disk and we deduce $\|z_t\|_\infty <1$, and then $\kappa <\infty$.

The differentiability of $t\mapsto z_t$ in supremum norm yields $\|v_{t,h}\|_\infty = \|z_{t+h} - z_t\|_\infty = O\bigl(h^2\bigr)$. Altogether $\|Q_{t,h}\| \leq \sqrt \kappa \|v_{t,h}\|_\infty^2 =O\bigl(h^2\bigr)$.

 A similar bound based on $v_{t,h} = h (1-z_t)^2 + O\bigl(h^2\bigr)$ shows that the linear (in $v_{t,h}$) term is
\[
	h\, \frac{1}{{\sqrt p}^n}\sum_{i=1}^n (1-z_t(x_i))^2 \prod_{j\neq i} z_t(x_j) + O\bigl(h^2\bigr).
\]
It follows that $t\mapsto \mathcal E_{z_t}$ is norm-differentiable with derivative
\begin{equation} \label{eq:ezdiff}
	\partial_t \mathcal E_{z_t}(\delta_{x_1}+\cdots +\delta_{x_n})
	= \frac{1}{\sqrt p} \sum_{i=1}^n \bigl(1-z_t(x_i)^2\bigr) \prod_{j\neq i } \frac{z_t(x_j)}{\sqrt p}
\end{equation}
for $n\geq 1$ and $\partial_t\mathcal E_{z_t}(0) =1$. The map $t\mapsto f_t$ is norm-differentiable with derivative
\begin{equation}\label{eq:ftdiff}
	\partial_t f_t = C'(t) \mathcal E_{z_t}+ C(t) \partial_t \mathcal E_{z_t}.
\end{equation}

(ii) We show that every exponential vector $\mathcal E_z$, $\|z\|_\infty <1$, is in the domain of the closure of $k^+(\1) - k^-(\1)$. Set $\psi_0 = \1_{\{0\}}$ and
\[
\psi_n(\delta_{x_1}+\cdots +\delta_{x_m}) = \delta_{n,m} \prod_{i=1}^m z(x_i)/\sqrt p
\]
 so that $\mathcal E_z = \sum_{n=0}^\infty \psi_n$. Every $\psi_n$ is in $\mathcal D$. Furthermore,
\begin{gather*}
	k^+(\1) \psi_n (\delta_{x_1}+\cdots +\delta_{x_m}) = \delta_{m,n+1}
		\frac 1 {\sqrt p} \sum_{i=1}^m \prod_{j\neq i} \frac{z(x_j)}{\sqrt p},
\\
	\sum_{n=0}^\infty \bigl\|k^+(\1) \psi_n\bigr\|^2 \leq \E\bigl[ \1_{\{\eta(E)\geq 1\}} p^{- \eta(E)} \eta(E)^2 \|z\|_\infty^{2(\eta(E) - 1)}\bigr]< \infty.
\end{gather*}
It follows that $k^+(\1) \sum_{n=0}^N \psi_n$ converges in norm. Turning to $k^-(\1)$, we note
\begin{align*}
	k^-(\1) \psi_n(\delta_{x_1}+\cdots +\delta_{x_m}) &{} = \delta_{m,n-1}\biggl( \int z \dd \alpha+ \sum_{i=1}^m z(x_i)\biggr) \prod_{j=1}^m \frac{z(x_i)}{\sqrt p}\\
	&{} = \delta_{m,n-1} \biggl( \int z \dd\alpha\biggr)\prod_{j=1}^m \frac{z(x_j)}{\sqrt p}
	+ \delta_{m,n-1} \frac{1}{\sqrt p} \sum_{i=1}^m z(x_i)^2 \prod_{j\neq i} \frac{z(x_j)}{\sqrt p}
\end{align*}
and $\sum_{n=1}^\infty \|k^-(\1) \psi_n\|^2 <\infty$ as well. Therefore, the partial sums $S_N = \sum_{n=0}^N \psi_n$ are in $\mathcal D$, they converge in norm to $\mathcal E_z$, and $\bigl(k^+(\1)-k^-(\1)\bigr)S_N$ converges in norm to \smash{$\sum_n \bigl(k^+(\1)-k^-(\1)\bigr)\psi_n$}. It follows that $\mathcal E_z$ is in the domain of the closure of \smash{$\bigl(k^+(\1) - k^-(\1)\bigr)$} and
\begin{align}
&	\bigl(k^+(\1)-k^-(\1)\bigr)^\mathrm{cl} \mathcal E_z(\delta_{x_1}+\cdots + \delta_{x_n})\nonumber\\
& \qquad{}	= - \biggl(\int z \dd \alpha\biggr) \mathcal E_z(\delta_{x_1}+\cdots + \delta_{x_n})
	 + \frac{1}{\sqrt p} \sum_{i=1}^n \bigl( 1- z(x_i)^2\bigr) \prod_{j\neq i} \frac{z(x_j)}{\sqrt p}.\label{eq:kpmft}
\end{align}
The function $f_t = C(t) \mathcal E_{z_t}$ is in the domain of the closure too.

(iii) Comparing the previous equation with~\eqref{eq:ftdiff}, \eqref{eq:ezdiff}, and~\eqref{eq:ctdiff}, we get $\partial_t f_t = \bigl(k^+(\1) - \smash{k^-(\1)\bigr)^\mathrm{cl}} f_t = \mathrm i A f_t$. The proof is complete.
\end{proof}

\begin{Remark}
 The proof can be concluded using an explicit formula given in \cite[Proposition~3.1]{lytvynov2003JFA} for generating functionals $\mathcal G_\varphi\colon \mathbf N_{<\infty}\to \C$ defined by
 \[
 	\mathcal G_\varphi(\eta)= 1+\sum_{n=1}^\infty \frac{1}{n!} (1-p)^n I_n(\varphi^{\otimes n})(\eta)
 \]
 for measurable $\varphi\colon E\to \C$ with $\sup |\varphi|<1$:
\[
	\mathcal G_\varphi(\eta) = \exp\biggl( - \int\log (1+ p \varphi) \dd \alpha + \int \log \biggl( \frac{1+ \varphi}{1+p\varphi}\biggr) \dd \eta\biggr).
\]
Applying Proposition~\ref{prop:expo-trafo} to exponential vectors $\mathcal E_{\varphi\sqrt p}$ with $\varphi$ an arbitrary bounded measurable function, we find that $(1-p)^{-\alpha(E)/2} \exp\bigl(\xi k^+(\1) - \overline{\xi} k^-(\1)\bigr)$ maps the function $f\colon\mathbf N_{<\infty}\to \C$ given by
\[
	f(0) =1,\qquad f(\delta_{x_1}+\cdots +\delta_{x_n}) = \prod_{i=1}^n \varphi(x_i)
\]
to $\mathcal G_\varphi$, from which we deduce that Theorem~\ref{thm:unitary-meixner} holds true when $f_n = \varphi^{\otimes n}$, much in the same way we prove Proposition~\ref{prop:meixner} below. The general statement follows with polarization formulas and (multi)linearity, and density arguments--notice that $(f_1,\ldots,f_n)\mapsto I_n(f_1\otimes_\mathrm s\cdots \otimes_\mathrm s f_n)$ is a~symmetric multilinear mapping and as such, uniquely determined by its value for $f_1=\cdots = f_n$.

For convenience of the reader, we provide a self-contained proof that does not require adaption of previously proven statements to our setup (such as \cite[Proposition~3.1]{lytvynov2003JFA} which uses the machinery of Jacobi fields and distribution theory).
\end{Remark}

We start by applying Proposition~\ref{prop:expo-trafo} to exponential vectors associated with indicators.

\begin{Corollary} \label{cor:u}
	Let $\xi>0$ with $\tanh \xi = \sqrt p$. Then, for every $B\in \mathcal E$ and $s\in (-1,1)$, the unitary $U = \exp\bigl(\xi k^+(\1)- \xi k^-(\1)\bigr)$ maps the function $f_s\colon \mathbf N_{<\infty}\to \R$ given by
	\[
		f_s(0) =1,\qquad f_s(\delta_{x_1}+\cdots + \delta_{x_n}) = s^n \prod_{i=1}^n \1_B(x_i)
	\]
	to
	\[
		Uf_s(\eta) = (1-p)^{\alpha(E)/2} ( 1 + p s )^{-\alpha (B)} \biggl( \frac{1+s}{1+ps}\biggr)^{\eta(B)}.
	\]
\end{Corollary}

\begin{proof}
	The function $f_s$ is the exponential state associated with $z(x) = s \sqrt p \1_B(x)$. The transformed $z$ is
	\[
		z_\xi(x) =\frac{s\sqrt p \1_B(x) + \sqrt p}{1+s p \1_B(x)} = \sqrt p\, \frac{1+s \1_B (x)}{1+ps\1_B (x)}
	\]
	and the associated exponential state is
	\[
		\mathcal E_{z_\xi}(\delta_{x_1}+\cdots +\delta_{x_n})
		= \prod_{i=1}^n \biggl( \frac{1+s \1_B (x_i)}{1+ps\1_B (x_i)}\biggr)^n
	\] 	
	or equivalently,
	\[
		\mathcal E_{z_\xi}(\eta)
		= \biggl( \frac{1+s}{1+ps}\biggr)^{\eta(B)}.
	\] 	
	The multiplier is
	\begin{align*}
	 C(\xi) = (\cosh \xi)^{-\alpha(E)} \exp\biggl( - \int \log (1+ s\sqrt p \1_B \tanh \xi )\dd\alpha\biggr)
	 = (1- p)^{\alpha(E)/2} ( 1+ ps)^{-\alpha(B)}.
	\end{align*}
	In the last line we have used $(\cosh \xi)^2 = 1/\bigl(1- (\tanh \xi)^2\bigr)$. We apply Proposition~\ref{prop:expo-trafo} and obtain the corollary.
\end{proof}

We remind 
that the univariate monic Meixner polynomials are denoted by $\mathcal{M}_n(x; \alpha, p)$, see Appendix~\ref{app:meixner} below.

\begin{Proposition} \label{prop:meixner}
	Fix $B\in \mathcal E$.
	The unitary $U$ maps the functions $f^{(n)}\colon \mathbf N_{<\infty}\to \R$, $n\in \N$, given by
	\[
		f^{(n)}(\delta_{x_1}+\cdots + \delta_{x_m}) = \begin{cases}
			\displaystyle \prod_{i=1}^m \1_B(x_i), & m =n,\\
			 0,& m\neq n
			\end{cases}
	\]
	to
	\[
		U f^{(n)} (\eta) = (1-p)^{\alpha(E)/2}\times \frac{1}{n!}\, (1-p)^n \mathcal M_n(\eta(B);\alpha(B),p).
	\]
\end{Proposition}

\begin{proof}
	Let $f^{(0)}:= \1_{\{0\}}$. For $s\in (-1,1)$, with the notation of Corollary~\ref{cor:u},
	\[
		f_s = \sum_{n=0}^\infty s^n f^{(n)}.
	\]
	The series converges in norm as the $f^{(n)}$'s are orthogonal and $\sum_{n=0}^\infty |s|^{2n} \bigl\|f^{(n)}\bigr\|^2<\infty$. As $U$ is unitary, hence continuous, we may exchange summation and the application of $U$ which gives
	\begin{equation} \label{eq:ufs1}
		U f_s = \sum_{n=0}^\infty s^n U f^{(n)}
	\end{equation}
	and the series on the right converges in norm. On the other hand, in the expression for $U f_s$ from Corollary~\ref{cor:u}, we recognize the generating function of Meixner polynomials with leading coefficient $(1-p)^n$:
	\begin{equation} \label{eq:ufs2}
		U f_s = (1-p)^{\alpha(E)/2} \sum_{n=0}^\infty \frac{s^n}{n!} (1-p)^n \mathcal M_n( \eta(B);\alpha(B),p),
	\end{equation}
	see equation~\eqref{eq:pmonic-gf}. The right side is not only absolutely convergent, pointwise for each $\eta\in \mathbf N_{<\infty}$, but also absolutely convergent in the Hilbert space $\mathcal H$. Indeed as $\eta\mapsto \eta(B)$, under $\rho_{p,\alpha}$ is a~negative binomial variable with parameters $\alpha(B)$ and $p$ (see equation~\eqref{eq:negbin}), the orthogonality relations for univariate Meixner polynomials recalled in Appendix~\ref{app:meixner} tell us that the summands are orthogonal and
	\[
		\frac{s^{2n}}{n!^2} (1-p)^{2n} \int|\mathcal M_n( \eta(B);\alpha(B),p)|^2 \rho_{p,\alpha}(\dd \eta)
		 = \frac{1}{n!} (\alpha(B))_n \bigl(p s^2\bigr)^{n}.
	\]
	The sum over $n\in \N_0$ of the right side is $\bigl(1-ps^2\bigr)^{\alpha(B)}$ and in particular, finite. Therefore, the series on the right side in~\eqref{eq:ufs2} converges in norm.
	
	Thus, we have two convergent series expansions for the analytic function $s\mapsto U f_s$ from $(-1,1)$ to $\mathcal H$. As the expansion coefficients are uniquely determined, the coefficients of $s^n$ in equations~\eqref{eq:ufs1} and~\eqref{eq:ufs2} must be the same. The proposition follows.
\end{proof}

The \emph{symmetric tensor product} of $f_1,\ldots, f_n\colon E\to \C$ is the function from $E^n$ to $\C$ given by
\[
	f_1\otimes_\mathrm s \cdots \otimes_\mathrm s f_n(x_1,\ldots,x_n) = \frac{1}{n!}\sum_{\sigma\in \mathfrak S_n} f_{\sigma(1)}(x_1)\cdots f_{\sigma(n)}(x_n).
\]
Above, $\mathfrak{S}_n$ denotes the set of permutations of the numbers $\set{1, \ldots, n}$.
Consider disjoint sets $B_1,\ldots, B_\ell$, integers $n_1,\ldots, n_\ell\in \N$ adding up to $\N$, and $f_n =\1_{B_1}^{\otimes_\mathrm s n_1} \otimes_\mathrm s \cdots \otimes_\mathrm s \1_{B_\ell}^{\otimes_\mathrm s n_\ell}$. The associated function $f$ is supported on configurations that have exactly $n_i$ particles in $B_i$, for each $i=1,\ldots,\ell$, and $n$ particles in total. On the relevant configurations there is an additional combinatorial contribution. One finds
\begin{equation}\label{eq:tenso-indic}
	f(\eta) = \frac{n_1!\cdots n_\ell!}{n!} \1_{\{\eta(E\setminus \bigcup_{i=1}^\ell B_i)=0\}} \prod_{i=1}^\ell\1_{\{\eta(B_i) = n_i\}}.
\end{equation}
Set $B_0:= E\setminus \bigl(\bigcup_{i=1}^\ell B_i\bigr)$, $n_0=0$, and
\begin{equation} \label{eq:fi-indic}
	f_i(\eta) = \1_{\{\eta(B_i)=n_i \}},\qquad i =0,\ldots,\ell.
\end{equation}
The next lemma says that $U$ maps the product of $f_i$'s to products of $U f_i$'s.

\begin{Lemma} \label{lem:tensorize}
	Let $\{B_0,\ldots, B_{\ell}\}$ be a set partition of $E$, $n_0,n_1,\ldots,n_\ell \in \N_0$, and $f_0,\ldots,f_\ell$ as in equation~\eqref{eq:fi-indic}. Then
	\[
		( U (f_0\cdots f_\ell)) (\eta) =\prod_{j=0}^\ell \frac{1}{n_j!} (1-p)^{n_j+ \alpha(B_j)/2 } \mathcal M_{n_j}( \eta(B_j);\alpha(B_j), p).
	\] 	
\end{Lemma}

\begin{proof}
	Let $V(t):= \exp\bigl( t \bigl(k^+(\1) - k^-(\1)\bigr)\bigr)$. Notice that
	\[
		k^+(\1) - k^-(\1) = \sum_{j=0}^\ell \bigl( k^+(\1_{B_j}) - k^-(\1_{B_j})\bigr).
	\]
	For $j=0,\ldots,\ell$, let $\nu_j$ be the negative binomial law on $\N_0$ with parameters $p$ and $\alpha(B_j)$. The space $d_0$ of sequences with at most finitely many non-zero entries is dense in $\ell^2(\N_0,\nu_j)$. Consider the operators $k_j^\pm\colon d_0\to \ell^2(\N_0,\nu_j)$ given by
	\[
		k_j^+ g_j(n) = \frac{1}{\sqrt p} n g_{j}(n -1),\qquad k_i^- g_i(n) = \sqrt{p} \bigl(\alpha(B_{i})+n\bigr) g_i(n+1).
	\]
	The operator $\frac{ 1} {\mathrm i} \bigl(k_j^+- k_j^-\bigr)$ is closable with self-adjoint closure, let $U_j(t) = \exp\bigl( t \bigl(k_j^+-k_j^-\bigr)\bigr)$, $t\in \R$ be the associated strongly continuous unitary group.
	For $g_0,\ldots, g_\ell \in d_0$ and $t\in \R$, set
	\[
		G_t(\eta) = \prod_{j=0}^\ell ( U_j(t) g_j)( \eta(B_j)).
	\]
	The map $t\mapsto G_t$ from $\R\to \mathcal H$ is norm-differentiable, $G_t$ is in the domain of the closure $L$ of $k^+(\1)-k^-(\1)$, and $\frac{\dd}{\dd t} G_t(\eta) = L G_t$. It follows that $G_t = V(t) G_0$. We apply the statement to $g_j(n) = \1_{\{n=n_j\}}$ and $t=\xi$ and deduce
	\begin{equation} \label{eq:tenso1}
		{\rm U}(f_0\cdots f_\ell) (\eta) = \prod_{j=0}^\ell ( U_j(\xi) g_j)( \eta(B_j)).
	\end{equation}
	Finally, we notice that each term on the right side is given by a univariate Meixner polynomial~as%
	\begin{equation} \label{eq:tenso2}
		( U_j(\xi) g_j)(m) = \frac{1}{n_j!}\, (1-p)^{n_j+\alpha(B_j)/2} \mathcal M_{n_j}(m;\alpha(B_j),p).
	\end{equation}
	This follows from Proposition~\ref{prop:meixner} applied to a set $E'$ that is a singleton $\{x_0\}$, to $B= \{x_0\}$, and the measure $\alpha'$ that gives mass $\alpha(B_j)$ to $x$. Alternatively, one may deduce the univariate equation~\eqref{eq:tenso2} directly from Theorem 3.1 in Carinci et al.~\cite{carinci2019orthogonal}.
	To conclude, we insert~\eqref{eq:tenso2} into~\eqref{eq:tenso1} and obtain the lemma.
\end{proof}

\begin{proof}[Proof of Theorem~\ref{thm:unitary-meixner}]
	Let $B_1,\ldots, B_\ell \in \mathcal E$ disjoint, $n_1,\ldots,n_\ell \in \N$ with $n_1+\cdots + n_\ell=n$, and $f\colon \mathbf N_{<\infty}\to \C$ the function with $f(\delta_{x_1}+\cdots + \delta_{x_n}) = \1_{B_1}^{\otimes_\mathrm s n_1}\otimes_{\mathrm s} \cdots \otimes_{\mathrm s} \1_{B_\ell}^{\otimes_\mathrm s n_\ell}(x_1,\ldots,x_n)$. Set $B_0 = E\setminus \bigl(\bigcup_{j=1}^\ell B_j\bigr)$ and $n_0=0$. Equation~\eqref{eq:tenso-indic} and Lemma~\ref{lem:tensorize} yield
	\begin{align*}
		U f &{}= \frac{n_1!\cdots n_\ell!}{n!}\, (1-p)^{\alpha(B_0)/2} \times \prod_{j=1}^\ell \frac{1}{n_j!} (1-p)^{n_j+\alpha(B_j)/2}\mathcal M_{n_j}( \eta(B_j);\alpha(B_j),p) \\
				&{}= (1-p)^{\alpha(E)/2} \times \frac{1}{n!}(1-p)^n \prod_{j=1}^\ell \mathcal M_{n_j}( \eta(B_j);\alpha(B_j),p) \\
				&{}= (1-p)^{\alpha(E)/2} \times \frac{1}{n!}(1-p)^n I_n(f_n).
	\end{align*}
	This proves the theorem for all functions $f$ of the form~\eqref{eq:tenso-indic} with disjoint $B_1,\ldots, B_\ell$. By linearity, the statement extends to all indicators of sets of the form
	\[
		\{\eta\colon \eta(E) = n\}\cap \Biggl( \bigcap_{j=1}^\ell \{\eta\colon \eta(B_j) = n_j\}\Biggr)
	\]
	with $\ell \in \N$, $n_1,\ldots,n_\ell\in \N_0$, and $B_1,\ldots, B_\ell \in \mathcal E$ (not necessarily disjoint). These sets form a~generating $\pi$-system of the $\sigma$-algebra $\mathcal N_n$ on $\mathbf N_n =\{\eta\colon\eta(E) =n\}$ that is the trace of $\mathcal N_{<\infty}$ on~$\mathbf N_n$.
	
	We conclude with a monotone class argument. Let $\mathcal V$ be the set of bounded functions $f\colon\mathbf N_\infty\to \C$ supported in $\mathbf N_n$ for which $U f = n!^{-1} (1-p)^{n+\alpha(E)/2} I_n(f_n)$. Clearly $\mathcal V$ is a~vector space. We check that it is closed under monotone pointwise limits of uniformly bounded sequences: Suppose $f^{(1)} \leq f^{(2)} \leq \ldots$ and $f^{(m)} \to f$ with $C = \sup_n\bigl\|f^{(m)}\bigr\|_\infty< \infty$ and $f^{(m)}\in \mathcal V$ for all $m$. Then, on the one hand $f^{(m)} \to f$ also in $L^2$-norm and $\bigl\| Uf^{(m)} - U f\bigr\| \to 0$ because $U$ is unitary. On the other hand 	
	\begin{align*}
		\bigl\|I_n\bigl(f^{(m)}- f\bigr)\bigr\|^2 & \leq \int \biggl| \int \bigl(f^{(m)}(\vect x) - f(\vect x) \bigr) \eta^n(\dd \vect x) \biggr|^2 \rho_{p,\alpha}(\dd \eta)
	\end{align*}
	because orthogonal projections onto a subspace decrease the norm. The integrand $f^{(m)} -f$ in the inner integral on the right side goes to zero pointwise on $E^n$ and is bounded by $2C$ with $\int 2 C \dd \eta^{n} \leq 2 C\eta(E)^n<\infty$, therefore by dominated convergence it goes to zero. All moments of the negative binomial variable $\eta\mapsto \eta(E)$ are finite, therefore dominated convergence applied to the outer integral yields $\bigl\|I_n\bigl(f^{(m)}\bigr) - I_n(f)\bigr\|\to 0$ as $m\to \infty$. Thus, we may pass to the limit on both sides of the equality \smash{$U f^{(m)} = n!^{-1} (1-p)^{n+\alpha(E)/2} I_n\bigl(f^{(m)}_n\bigr)$} and we deduce $f\in \mathcal V$.
	
	We have shown that $\mathcal V$ is a vector space that contains all indicators of a $\pi$-system that generates the $\sigma$-algebra and that it is stable under pointwise limits of uniformly bounded sequences. The functional monotone class theorem \cite[Theorem 2.12.9]{bogachev-vol1} implies that it contains all bounded measurable functions.
\end{proof}

As an alternative route to obtain the above results, one could skip exponential states and Propositions~\ref{prop:expo-trafo} and~\ref{prop:meixner} and start directly with Lemma~\ref{lem:tensorize}, deducing the crucial equation~\eqref{eq:tenso2} directly from~\cite[Theorem 3.1]{carinci2019orthogonal}.

\appendix

\section{Univariate Meixner polynomials} \label{app:meixner}
We recall a few facts about the Meixner polynomials from Koekoek, Lesky and Swarttouw \mbox{\cite[Section 9]{HypergeometricOrthogonalPolynomials}}. The Meixner polynomials with parameters $\alpha>0$ and $p\in (0,1)$ are given by
\begin{equation} \label{eq:meixner-def}
 M_n(x) = M_n(x; \alpha; p) := \sum_{k=0}^\infty \frac{(-x)_k (-n)_k}{(\alpha)_k} \frac{1}{k!} \biggl(1- \frac1p\biggr)^k,
\end{equation}
where $(a)_0 = 1$ and $(a)_k = a(a+1)\cdots (a+k-1)$.
They satisfy the symmetry $M_n(x) = M_x(n)$ and the orthogonality relation
\[
	\sum_{x=0}^\infty M_n(x) M_m(x) (1-p)^\alpha w_{p,\alpha}(x) = \delta_{n,m} \frac{1}{w_{p,\alpha}(n)},\qquad w_{p,\alpha}(n) = \frac{p^n}{n!} (\alpha)_n.
\] 	
The generating function is
\[
	G(t,x) = \sum_{n=0}^\infty \frac{(\alpha)_n}{n!} t^n M_n(x) = \biggl( 1- \frac t p \biggr)^x (1-t)^{-\alpha - x}.
\]
The Meixner polynomials from~\eqref{eq:meixner-def} do not have leading coefficient $1$, instead they are related to the monic Meixner polynomials $\mathcal M_n(x) = \mathcal M_n(x; \alpha, p)$ by
\begin{equation*} 
	M_n(x; \alpha, p) = \frac{1}{(\alpha)_n} \biggl( 1- \frac 1p\biggr)^n \mathcal M_n(x; \alpha, p)
\end{equation*}
from which we get
\begin{equation} \label{eq:pmonic-gf}
	\sum_{n=0}^\infty \frac{s^n}{n!}\,(1-p)^n \mathcal M_n(x) = G(-ps, x) = (1+s)^x (1+p s)^{-x-\alpha}.
\end{equation}
The formula enters the proof of Proposition~\ref{prop:meixner}.

\subsection*{Acknowledgements}

We thank the anonymous referees for their careful reading and helpful suggestions that helped improve the article.
S.F.\ acknowledges financial support from the Engineering and Physical Sciences Research Council of the United Kingdom through the EPSRC Early Career Fellowship EP/V027824/1 and from the Deutsche Forschungsgemeinschaft (DFG, German Research Foundation) under Germany’s Excellence Strategy – EXC-2047/1 – 390685813. S.J.\ and S.W.\ were supported under Germany's excellence strategy EXC-2111-390814868. S.F.\ and S.W.\ thank the Hausdorff Institute for Mathematics (Bonn) for its hospitality during the Junior Trimester Program \textit{Stochastic modelling in life sciences} funded by the Deutsche Forschungsgemeinschaft (DFG, German Research Foundation) under Germany’s Excellence Strategy - EXC-2047/1 - 390685813. The authors thank F. Redig for his insights at the beginning of the project. S.J.\ and S.W.\ thank T.~Kuna and E.~Lytvynov for helpful discussions.


\pdfbookmark[1]{References}{ref}
\LastPageEnding

\end{document}